\documentclass[11pt]{article}
\usepackage{amsthm}
\usepackage{amsfonts,amssymb}
\usepackage{amsmath}
\usepackage{colonequals}
\usepackage{lmodern}
\usepackage[utf8]{inputenc}
\newtheorem{conjecture}{Conjecture}
\newtheorem{theorem}{Theorem}

\newtheorem{lemma}[theorem]{Lemma}

\newcommand{\sst}[2]{\left\{#1\,:\,#2\right\}}

\newcommand{\R}{\mathbf{R}}
\renewcommand{\le}{\leqslant}
\renewcommand{\ge}{\geqslant}

\title{Fractional coloring of triangle-free planar graphs\thanks{This research was supported by the Czech-French Laboratory LEA STRUCO.}}
\author{Zdeněk Dvořák\thanks{Computer Science Institute of Charles University, Prague, Czech Republic.
E-mail: \texttt{rakdver@iuuk.mff.cuni.cz}.
Supported by the Center of Excellence -- Inst. for Theor. Comp. Sci., Prague, project P202/12/G061 of Czech Science Foundation.}
\and
Jean-Sébastien Sereni\thanks{\emph{Centre National de la Recherche
Scientifique} (LORIA), Nancy, France. E-mail: \texttt{sereni@kam.mff.cuni.cz}.
This author's work was partially
supported by the French \emph{Agence Nationale de la Recherche} under reference
\textsc{anr 10 jcjc 0204 01}.}
\and
Jan Volec\thanks{Mathematics Institute and DIMAP, University of Warwick, Coventry CV4 7AL, UK.
E-mail: \texttt{honza@ucw.cz}. This author's work was supported by a grant of the French Government.}}
\date{}

\begin{document}
\maketitle

\begin{abstract}
We prove that every planar triangle-free graph on $n$ vertices has fractional chromatic number at most $3-\frac{1}{n+1/3}$.
\end{abstract}

\section{Introduction} 
\label{sec:intro}
Coloring of triangle-free planar graphs is an attractive topic.  It started
with Grötzsch's theorem~\cite{grotzsch1959}, stating that such graphs are
$3$-colorable. Since then, several simpler proofs have been given, e.g.,~by
Thomassen~\cite{thom-torus,ThoShortlist}. Algorithmic questions have also been addressed: while most proofs
readily yield quadratic algorithms to $3$-color such graphs, it takes
considerably more effort to obtain asymptotically faster algorithms. Kowalik~\cite{Kow3col}
proposed an algorithm running in time $O(n\log n)$, which relies on the design of an
advanced data structure. More recently, Dvořák \emph{et al.}~\cite{DvoKawTho}
managed to obtain a linear-time algorithm, yielding at the same time a yet
simpler proof of Grötzsch's theorem.

The fact that all triangle-free planar graphs admit a $3$-coloring implies
that all such graphs have an independent set containing at least one third of
the vertices. Albertson \emph{et al.}~\cite{AlbBolTuc} had conjectured that there is always a
larger independent set, which was confirmed by Steinberg and
Tovey~\cite{SteinbergTovey1993} even in a stronger sense:
all triangle-free planar $n$-vertex graphs admit a $3$-coloring where not all color classes
have the same size, and thus at least one of them forms an independent set of size
at least $\frac{n+1}{3}$.
This bound turns out to be tight for infinitely many triangle-free
graphs, as Jones~\cite{Jones1985} showed. As an aside, let us mention that
the graphs built by Jones have maximum degree $4$: this is no
coincidence as Heckman and Thomas later established that all triangle-free
planar $n$-vertex graphs with maximum degree at most $3$ have an independent
set of order at least $\frac{3n}{8}$, which again is a tight bound---actually
attained by planar graphs of girth $5$.

All these considerations naturally lead to investigate the fractional
chromatic number of triangle-free planar graphs. Indeed, as we shall later
see, this invariant actually corresponds to a weighted version of the
independence ratio. In addition, since $\chi_f(G)\le\chi(G)$ for every graph
$G$, Grötzsch's theorem implies that $\chi_f(G)\le3$ whenever $G$ is
triangle-free and planar.  On the other hand, Jones's construction shows the
existence of triangle-free planar graphs with fractional chromatic number
arbitrarily close to $3$.  Thus one wonders whether there exists a
triangle-free planar graph with fractional chromatic number exactly $3$. Let
us note that this happens for the circular chromatic number $\chi_c$, which is
a different relaxation of ordinary chromatic number such that $\chi_f(G)\le
\chi_c(G)\le \chi(G)$ for every graph $G$.

The purpose of this work is to answer this question. We do so by establishing
the following upper bound on the fractional chromatic number of
triangle-free planar $n$-vertex graphs, which depends on $n$.
\begin{theorem}\label{thm-main}
Every planar triangle-free graph on $n$ vertices has fractional chromatic
number at most $3-\frac{1}{n+1/3}$.
\end{theorem}
Consequently, no (finite) triangle-free planar graph has fractional
chromatic number equal to $3$.
We also note that the bound provided by Theorem~\ref{thm-main} is tight up to
the multiplicative factor.
Indeed, the aforementioned construction of Jones~\cite{Jones1985} yields,
for each $n\ge 2$ such that $n\equiv 2\pmod 3$, a triangle-free planar graph $G_n$
with $\alpha(G_n)=\frac{n+1}{3}$.
Consequently, $\chi_f(G_n)\ge \frac{3n}{n+1}=3-\frac{3}{n+1}$.

Our result can be improved for triangle-free planar graphs with maximum degree at most four,
giving an exact bound for such graphs.
\begin{theorem}\label{thm-main4}
Every planar triangle-free $n$-vertex graph of maximum degree at most four
has fractional chromatic number at most $\frac{3n}{n+1}$.
\end{theorem}

Furthermore, the graphs of Jones's construction contain a large number of separating 4-cycles
(actually, all their faces have length five).  We show that planar triangle-free graphs
of \emph{maximum degree $4$} and \emph{without separating $4$-cycles} cannot have fractional
number arbitrarily close to $3$.

\begin{theorem}\label{thm-main4nos}
There exists $\delta>0$ such that every planar triangle-free graph of maximum degree at most four
and without separating $4$-cycles has fractional chromatic number at most $3-\delta$.
\end{theorem}

Dvořák and Mnich~\cite{dmnich} proved that there exists $\beta>0$ such that all planar triangle-free
$n$-vertex graphs without separating $4$-cycles contain an independent set of size at least $n/(3-\beta)$.
This gives an evidence that the restriction on the maximum degree in Theorem~\ref{thm-main4nos} might not
be necessary.

\begin{conjecture}\label{conj-h}
There exists $\delta>0$ such that every planar triangle-free graph without separating $4$-cycles has
fractional chromatic number at most $3-\delta$.
\end{conjecture}

Faces of length four are usually easy to deal with in the proofs by collapsing; thus the following
seemingly simpler variant of Conjecture~\ref{conj-h} is likely to be equivalent to it.

\begin{conjecture}[Dvořák and Mnich~\cite{dmnich}]
There exists $\delta>0$ such that every planar graph of girth at least five has
fractional chromatic number at most $3-\delta$.
\end{conjecture}

%
\section{Notation and auxiliary results} 
\label{sec:note}
Let $\mu$ be the Lebesgue measure on real numbers.  Let $G$ be a graph.  If a
function $\varphi$ assigns to each vertex of $G$ a measurable subset of
$[0,1]$ and $\varphi(u)\cap \varphi(v)=\emptyset$ for all edges $uv$ of $G$,
we say that $\varphi$ is a \emph{fractional coloring} of $G$.  Let $f\colon
V(G)\to Q\cap [0,1]$ be a function with rational values.  If
the fractional coloring $\phi$ satisfies
$\mu(\varphi(v))\ge f(v)$ for every $v\in V(G)$, then we say that $\varphi$ is
an \emph{$f$-coloring} of $G$.  If $\mu(\varphi(v))=f(v)$ for every $v\in
V(G)$, then we say that $\varphi$ is a \emph{tight $f$-coloring}.  Note that
if $G$ has an $f$-coloring, then it also has a tight one.  For $x\in Q\cap
[0,1]$, let $c_x$ denote the constant function assigning the value $x$ to each
vertex of $G$.  The \emph{fractional chromatic number} of $G$ is defined as
\[
\chi_f(G)=\frac{1}{\sup \sst{x\in Q\cap [0,1]}{\text{$G$ has a $c_x$-coloring}}}.
\]
Let $w\colon V(G)\to \R^+$ be an arbitrary function.  For a set $X\subseteq
V(G)$, by $w(X)$ we mean $\sum_{v\in X} w(v)$.  Let $w(f) = \sum_{v\in V(G)}
f(v)w(v)$.  An integer $N\ge 1$ is a \emph{common denominator} of $f$ if
$Nf(v)$ is an integer for every $v\in V(G)$.  Setting $[N]=\{1,\ldots,N\}$,
a function $\psi\colon V(G)\to
\mathcal{P}([N])$ is an \emph{$(f,N)$-coloring} of $G$ if $\psi(u)\cap
\psi(v)=\emptyset$ for every $uv\in E(G)$ and $|\psi(v)|\ge Nf(v)$ for every
$v\in V(G)$.  The $(f,N)$-coloring is \emph{tight} if $|\psi(v)|=Nf(v)$ for
every $v\in V(G)$.

The fractional chromatic number of a graph can be expressed in various equivalent
ways, based on its well known linear programming formulation and duality.
The proof of the following lemma can be found e.g. in Dvořák \emph{et al.~}{\cite[Theorem~2.1]{fracsub}}.
\begin{lemma}\label{lemma-eqw}
Let $G$ be a graph and $f\colon V(G)\to Q\cap [0,1]$ a function.  The following
statements are equivalent.
\begin{itemize}
\item The graph $G$ has an $f$-coloring.
\item There exists a common denominator $N$ of $f$ such that $G$ has an $(f,N)$-coloring.
\item For every $w\colon V(G)\to \R^+$, there exists an independent set $X\subseteq V(G)$ with $w(X)\ge w(f)$.
\end{itemize}
\end{lemma}

We need several results related to Gr\"otzsch's theorem.
The following lemma was proved for vertices of degree at most three by
Steinberg and Tovey~\cite{SteinbergTovey1993}.
The proof for vertices of degree four follows from the results of Dvořák and Lidický~\cite{col8cyc},
as observed by Dvořák \emph{et al.}~\cite{trfree5}.
\begin{lemma}\label{lemma-tovey}
If $G$ is a triangle-free planar graph and $v$ is a vertex of $G$ of degree at most four,
then there exists a $3$-coloring of $G$ such that all neighbors of $v$ have the same color.
\end{lemma}
\noindent In fact, Dvořák \emph{et al.}~\cite{trfree5} proved the following stronger statement.
\begin{lemma}\label{lemma-far}
There exists an integer $D\ge 4$ with the following property.
Let $G$ be a triangle-free planar graph without separating $4$-cycles and
let $X$ be a set of vertices of $G$ of degree at most four.  If the distance between every two vertices in $X$ is at least $D$,
then there exists a $3$-coloring of $G$ such that all neighbors of vertices of $X$ have the same color.
\end{lemma}

Let $G$ be a triangle-free plane graph.  A $5$-face $f=v_1v_2v_3v_4v_5$ of $G$ is \emph{safe}
if $v_1$, $v_2$, $v_3$ and $v_4$ have degree exactly three, their
neighbors $x_1$, \ldots, $x_4$ (respectively) not incident with $f$ are
pairwise distinct and non-adjacent, and
\begin{itemize}
\item the distance between $x_2$ and $v_5$ in $G-\{v_1,v_2,v_3,v_4\}$
is at least four, and
\item $G-\{v_1,v_2,v_3,v_4\}$ contains no path of length exactly
three between $x_3$ and $x_4$.
\end{itemize}

\begin{lemma}[Dvořák \emph{et al.}~{\cite[Lemma~2.2]{DvoKawTho}}]\label{lemma-safe}
If $G$ is a plane triangle-free graph of minimum degree at least three
and all faces of $G$ have length five, then $G$ has a safe face.
\end{lemma}

Finally, let us recall the folding lemma, which is frequently used in the
coloring theory of planar graphs.

\begin{lemma}[Klostermeyer and Zhang~\cite{KloZhang}]\label{lemma-folding}
Let $G$ be a planar graph with odd-girth $g>3$.
If $C=v_0v_1\dots v_{r-1}$ is a facial circuit of $G$ with $r\ne g$,
then there is an integer $i\in\{0,\ldots,r-1\}$ such that
the graph $G'$ obtained from $G$ by identifying $v_{i-1}$ and $v_{i+1}$
(where indices are taken modulo $r$) is also of odd-girth $g$.
\end{lemma}

\section{Proofs}

First, let us show a lemma based on the idea of Hilton \emph{et. al.}~\cite{planfr5}.

\begin{lemma}\label{lemma-wtone}
Let $G$ be a planar triangle-free graph and let $w:V(G)\to R^+$ be an arbitrary function.
If $v\in V(G)$ has degree at most $4$, then $G$ contains an independent set $X$ such that
$w(X)\ge \frac{w(V(G))+w(v)}{3}$.
\end{lemma}
\begin{proof}
Lemma~\ref{lemma-tovey} implies that there exists a $3$-coloring of $G$ such that all neighbors of $v$ have the
same color.  Consequently, $G$ has an $f_v$-coloring for the function $f_v$
such that $f_v(z)=1/3$ for $z\in V(G)\setminus \{v\}$ and $f_v(v)=2/3$.
By Lemma~\ref{lemma-eqw}, there exists an independent set $X\subseteq V(G)$
such that $w(X)\ge w(f_v)=\frac{w(V(G))+w(v)}{3}$.
\end{proof}

Theorem~\ref{thm-main4} now readily follows.
\begin{proof}[Proof of Theorem~\ref{thm-main4}]
Let $G$ be a planar triangle-free $n$-vertex graph of maximum degree at most
four.  Consider any function $w\colon V(G)\to \R^+$, and let $v$ be the vertex to
which $w$ assigns the maximum value.  We have $w(v)\ge w(V(G))/n$.  By Lemma~\ref{lemma-wtone},
there exists an independent set $X$ such that $w(X)\ge \frac{w(V(G))+w(v)}{3}\ge \frac{n+1}{3n}w(V(G))$.  Therefore, for every
$w\colon V(G)\to \R^+$, there exists an independent set $X$ with $w(X)\ge
w(c_{(n+1)/(3n)})$.  By Lemma~\ref{lemma-eqw}, it follows that the
fractional chromatic number of $G$ is at most $\frac{3n}{n+1}$.
\end{proof}

Similarly, Lemma~\ref{lemma-far} implies Theorem~\ref{thm-main4nos}.
\begin{proof}[Proof of Theorem~\ref{thm-main4nos}]
Let $D$ be the constant of Lemma~\ref{lemma-far},
let $\delta_0=\frac{1}{3\cdot 4^D}$ and $\delta=\frac{9\delta_0}{3\delta_0+1}=\frac{3}{4^D+1}$.  Let $G$ be a planar triangle-free graph of maximum degree at most four and without separating $4$-cycles.
Clearly, it suffices to prove that $G$ has a $c_{1/3+\delta_0}$-coloring.

Let $G'$ be the graph obtained from $G$ by adding edges
between all pairs of vertices at distance at most $D-1$.  The maximum degree of $G'$ is less than $4^D$,
and thus $G'$ has a coloring by at most $4^D$ colors.  Let $C_1$, \ldots, $C_{4^D}$ be the color classes of this coloring.
For $i\in \{1,\ldots, 4^D\}$, let $f_i$ be the function defined by $f_i(v)=2/3$ for $v\in C_i$ and $f_i(v)=1/3$ for $v\in V(G)\setminus C_i$.
Lemma~\ref{lemma-far} ensures that $G$ has an $f_i$-coloring.

Consider any function $w:V(G)\to R^+$.  There exists $i\in \{1,\ldots, 4^D\}$ such that $w(C_i)\ge w(V(G))/4^D$.
By Lemma~\ref{lemma-eqw} applied for $f_i$, we conclude that $G$ contains an independent set $X$ such that
$w(X)\ge w(f_i)=\frac{w(V(G))+w(C_i)}{3}\ge (1/3+\delta_0)w(V(G))=w(c_{1/3+\delta_0})$.
Since the choice of $w$ was arbitrary, Lemma~\ref{lemma-eqw} implies that $G$ has a $c_{1/3+\delta_0}$-coloring.
\end{proof}

The proof of Theorem~\ref{thm-main} is somewhat more involved.
Let $\varepsilon=1/9$ and for $n\ge 1$, let $b(n)=1/3+\varepsilon/n$.  Let $G$
be a plane triangle-free graph.  We say that $G$ is a \emph{counterexample} if
$G$ does not have a $c_{b(|V(G)|)}$-coloring.  We say that $G$ is a
\emph{minimal counterexample} if $G$ is a counterexample and no plane
triangle-free graph with fewer than $|V(G)|$ vertices is a counterexample.
Since $b$ is a decreasing function, every minimal counterexample is connected.
\begin{lemma}\label{lemma-mindeg2}
If $G$ is a minimal counterexample, then $G$ is $2$-connected.  Consequently, the minimum degree of $G$ is at least two.
\end{lemma}
\begin{proof}
Since $b(n)\le 1/2$, every counterexample has at least three vertices; hence, it suffices to prove that $G$ is
$2$-connected, and the bound on the minimum degree will follow.
Let $n$ be the number of vertices of $G$.

Suppose that $G$ is not $2$-connected, and let $G_1$ and $G_2$ be subgraphs of
$G$ such that $G=G_1\cup G_2$, the graph $G_1$ intersects $G_2$ in exactly one
vertex $v$, and both $n_1=|V(G_1)|$ and $n_2=|V(G_2)|$ are greater than $1$.
Since $n=n_1+n_2-1$, we have $n_1,n_2<n$, and thus neither $G_1$ nor $G_2$ is
a counterexample.  Consequently, $G_i$ has a $c_{b(n_i)}$-coloring for
$i\in\{1,2\}$.  Since $b$ is a decreasing function, we deduce that $G_i$ has a $c_{b(n)}$-coloring and hence,
by Lemma~\ref{lemma-eqw}, there exists $N\ge 1$ such that $G_i$ has a $(c_{b(n)},N)$-coloring $\varphi_i$.
By permuting the colors if necessary, we can assume that $\varphi_1(v)=\varphi_2(v)$, and
thus $\varphi_1\cup\varphi_2$ is a $(c_{b(n)}, N)$-coloring of $G$.  This
contradicts the assumption that $G$ is a counterexample.
\end{proof}

\begin{lemma}\label{lemma-faces}
If $G$ is a minimal counterexample, then every face of $G$ has length exactly $5$.
\end{lemma}
\begin{proof}
Let $n$ be the number of vertices of $G$.
Suppose that $G$ has a face $f$ of length other than $5$.  Since $G$ is triangle-free, it has odd girth
at least five, and by Lemma~\ref{lemma-folding}, there exists a path $v_1v_2v_3$ in the boundary of $f$
such that the graph $G'$ obtained by identifying $v_1$ with $v_3$ to a single vertex $z$ has odd girth at least five as well.
It follows that $G'$ is triangle-free.  Since $G$ is a minimal counterexample, $G'$ has a $c_{b(n-1)}$-coloring,
and by giving both $v_1$ and $v_3$ the color of $z$, we obtain a $c_{b(n-1)}$-coloring of $G$.  Since $b(n)<b(n-1)$,
this contradicts the assumption that $G$ is a counterexample.
\end{proof}

Given a counterexample $G$ on $n$ vertices, a function $w\colon V(G)\to \R^+$
is a \emph{witness} if $G$ has no independent set $X$ satisfying $w(X)\ge
w(c_{b(n)})$.  By Lemma~\ref{lemma-eqw}, every counterexample has a witness.  Let us now state a useful
special case of Lemma~\ref{lemma-wtone}.
\begin{lemma}\label{lemma-witwt}
If $G$ is a counterexample on $n$ vertices, $w$ is a witness and $v\in V(G)$ has degree at most three, then $w(v)<3\varepsilon w(V(G))/n$.
\end{lemma}
\begin{proof}
Let $n$ be the number of vertices of $G$.  By Lemma~\ref{lemma-wtone}, there exists an independent set $X\subseteq V(G)$
with $w(X)\ge \frac{w(V(G))+w(v)}{3}$.  On the other hand, since $w$ is a witness, we have $w(X)<w(c_{b(n)})=\frac{w(V(G))}{3}+\frac{\varepsilon}{n}w(V(G))$.
The claim of this lemma follows.
\end{proof}

\begin{lemma}\label{lemma-mindeg3}
If $G$ is a minimal counterexample, then $G$ has minimum degree at least three.
\end{lemma}
\begin{proof}
Let $n$ be the number of vertices of $G$ and let $w\colon V(G)\to \R^+$ be a
witness for $G$.  By Lemma~\ref{lemma-mindeg2}, the graph $G$ has minimum degree at
least two.  Suppose that $v\in V(G)$ has degree two.  By
Lemma~\ref{lemma-witwt}, we have $w(v)<3\varepsilon w(V(G))/n$.

Since $G$ is a minimal counterexample, there exists $N\ge 1$ and a tight
$(c_{b(n-1)}, N)$-coloring $\psi$ of $G-v$.  Let $f(x)=b(n-1)$ for $x\in
V(G-v)$ and $f(v)=1-2b(n-1)$.  Clearly, $\psi$ extends to an $(f,N)$-coloring
of $G$.  By Lemma~\ref{lemma-eqw}, there exists an independent set $X\subseteq
V(G)$ such that
\begin{align*}
w(X)&\ge w(f)\\
&= b(n-1)w(V(G))-(3b(n-1)-1)w(v)\\
&> b(n-1)w(V(G))-\frac{3(3b(n-1)-1)\varepsilon}{n} w(V(G))\\
&= \left[b(n-1)-\frac{9\varepsilon^2}{n(n-1)}\right]w(V(G))\\
&= \left[b(n)+\frac{\varepsilon}{n(n-1)}-\frac{9\varepsilon^2}{n(n-1)}\right]w(V(G))\\
&= b(n)w(V(G))=w(c_{b(n)}).
\end{align*}
This contradicts that $w$ is a witness for $G$.
\end{proof}

\begin{lemma}\label{lemma-nosafe}
No minimal counterexample contains a safe $5$-face.
\end{lemma}
\begin{proof}
Let $G$ be a minimal counterexample containing a safe $5$-face
$f=v_1v_2v_3v_4v_5$, and let $x_1$, \dots, $x_4$ be the neighbors of $v_1$,
\ldots, $v_4$, respectively, that are not incident with $f$.  Let $n$ be the
number of vertices of $G$ and let $w\colon V(G)\to \R^+$ be a witness for $G$.
By Lemma~\ref{lemma-witwt}, we have $w(v_i)<3\varepsilon w(V(G))/n$ for $1\le
i\le 4$.

Let $G'$ be the graph obtained from $G-\{v_1,v_2,v_3,v_4\}$ by identifying
$x_2$ with $v_5$ into a new vertex $u_1$, and $x_3$ with $x_4$ into a new vertex $u_2$.
Since $f$ is safe, $G'$ is
triangle-free.  Since $G$ is a minimal counterexample, there exists $N\ge 1$
and a tight $(c_{b(n-6)}, N)$-coloring $\psi$ of $G'$.  Let $f(x)=b(n-6)$ for
$x\in V(G-\{v_1,v_2,v_3,v_4\})$ and $f(v_i)=1-2b(n-6)$ for $1\le i\le 4$.
We use $\psi$ to design an $(f,N)$-coloring of $G$.

Let $\psi(x_2)=\psi(v_5)=\psi(u_1)$ and $\psi(x_3)=\psi(x_4)=\psi(u_2)$.
Let $\psi(v_1)$ be a subset of $[N]\setminus (\psi(x_1)\cup \psi(v_5))$ of size
$f(v_1)N$, and let $\psi(v_2)$ be a subset of $[N]\setminus (\psi(x_2)\cup
\psi(v_1))$ of size $f(v_2)N$.  Let $M_3=[N]\setminus (\psi(v_2)\cup
\psi(x_3))$ and $M_4=[N]\setminus (\psi(v_5)\cup \psi(x_4))$.  Note that
$|M_3|\ge f(v_3)N$ and $|M_4|\ge f(v_4)N$.  Furthermore, since
$\psi(x_3)=\psi(x_4)$ and $\psi(v_2)\cap \psi(v_5)=\emptyset$, we have
$|M_3\cup M_4|=1-|\psi(x_3)|=1-b(n-6)\ge f(v_3)+f(v_4)$.  Therefore, we can
choose disjoint sets $\psi(v_3)\subseteq M_3$ and $\psi(v_4)\subseteq M_4$ of
size $f(v_3)N=f(v_4)N$.  This gives an $(f,N)$-coloring of $G$.

By Lemma~\ref{lemma-eqw}, there exists an independent set $X\subseteq V(G)$ such that
\begin{align*}
w(X)&\ge w(f)\\
&= b(n-6)w(V(G))-(3b(n-6)-1)\sum_{i=1}^4w(v_i)\\
&> b(n-6)w(V(G))-\frac{12(3b(n-6)-1)\varepsilon}{n} w(V(G))\\
&= \left[b(n-6)-\frac{36\varepsilon^2}{n(n-6)}\right]w(V(G))\\
&= \left[b(n)+\frac{6\varepsilon}{n(n-6)}-\frac{36\varepsilon^2}{n(n-6)}\right]w(V(G))\\
&\ge b(n)w(V(G))=w(c_{b(n)}).
\end{align*}
This contradicts that $w$ is a witness for $G$.
\end{proof}

We can now establish Theorem~\ref{thm-main}.
\begin{proof}[Proof of Theorem~\ref{thm-main}]
Note that $\frac{1}{3-\frac{1}{n+1/3}}=b(n)$.  Suppose that there exists a
planar triangle-free graph $G$ on $n$ vertices with fractional chromatic
number greater than $3-\frac{1}{n+1/3}$.  Then $G$ has no $c_{b(n)}$-coloring,
and thus $G$ is a counterexample.  Therefore, there exists a minimal
counterexample $G_0$. Lemmas~\ref{lemma-mindeg3}, \ref{lemma-faces} and
\ref{lemma-safe} imply that $G_0$ has a safe $5$-face.  However, that contradicts
Lemma~\ref{lemma-nosafe}.
\end{proof}

\end{document}